\documentclass[reqno]{amsart}

\usepackage{amsmath}
\usepackage{amssymb}
\usepackage{verbatim}
\usepackage{enumerate}
\usepackage[all]{xy}
\usepackage{amsthm}
\usepackage{amscd}
\usepackage{amsfonts}
\usepackage{latexsym}
\usepackage{amscd}
\usepackage{graphicx}
\usepackage{tikz}
\usepackage{setspace}
\usepackage{hyperref}
\usetikzlibrary{decorations.pathmorphing}

 \definecolor{darkgreen}{HTML}{336633}
 \definecolor{darkred}{HTML}{993333}
\makeatletter

\newcommand{\arxiv}[1]{\href{http://arxiv.org/abs/#1}{\tt
    arXiv:\nolinkurl{#1}}}

\def\bd#1{\text{\boldmath${#1}$}}
\newcommand{\jontodo}{\todo[inline,color=green!20]}

\newtheorem{theorem}{Theorem}[section]
\newtheorem{lemma}[theorem]{Lemma}

\newtheorem{corollary}[theorem]{Corollary} 
\theoremstyle{definition}  
\newtheorem{definition}[theorem]{Definition}

\@addtoreset{equation}{section}
\def\theequation{\arabic{section}.\arabic{equation}}

\newcommand{\End}{\operatorname{End}}
\newcommand{\wt}{\operatorname{wt}}
\newcommand{\Aut}{\operatorname{End}} 
\newcommand{\CatHom}{\mathcal{H}om}
\newcommand{\Hom}{\operatorname{Hom}} 
\newcommand{\HOM}{\operatorname{HOM}} 
\newcommand{\Mod}{\operatorname{-mod}}
\renewcommand{\mod}{\operatorname{-mod}_{l\!f}}
\newcommand{\Rep}{\operatorname{Rep}}
\newcommand{\id}{\text{id}}
\newcommand{\gl}{\mathfrak{gl}}
\newcommand{\g}{\mathfrak{g}}
\newcommand{\Z}{\mathbb{Z}}
\newcommand{\N}{\mathbb{N}}
\newcommand{\C}{\mathbb{C}}
\newcommand{\Q}{\mathbb{Q}}
\renewcommand{\k}{\Bbbk}
\newcommand{\eps}{\varepsilon}
\newcommand{\gr}{\operatorname{gr}}
\newcommand{\unit}{\mathds{1}}
\newcommand{\rev}{^{\operatorname{rev}}}
\newcommand{\rad}{{\operatorname{rad}}}
\def\la{\lambda}
\def\al{\alpha}
\def\be{\beta}
\def\ga{\gamma}
\def\Ga{\Gamma}
\def\De{\Delta}
\def\blob{{\pmb\circ}}
\def\U{{\mathcal U}}
\def\T{{\mathtt T}}

\allowdisplaybreaks

\begin{document}

\title[Kac-Moody $2$-categories]{\boldmath On the definition of Kac-Moody $2$-category}

\author[J. Brundan]{Jonathan Brundan}
\address{Department of Mathematics,
University of Oregon, Eugene, OR 97403, USA}
\email{brundan@uoregon.edu}

\thanks{2010 {\it Mathematics Subject Classification}: 17B10, 18D10.}
\thanks{Research
supported in part by NSF grant DMS-1161094.}

\begin{abstract}
We show that
the Kac-Moody $2$-categories defined by Rouquier and by Khovanov and Lauda are the same.
\end{abstract}

\maketitle  
\section{Introduction}

Assume that we are given
the following data:
\begin{itemize}
\item a (not necessarily finite) index set $I$;
\item integers $a_{ij}$ for each $i,j \in I$ such that
$a_{ii}  = 2$, $a_{ij} \leq 0$ for all $i \neq j$, and
$a_{ij} = 0$ if and only if $a_{ji} = 0$;
\item positive integers 
$d_i$ such that $d_i a_{ij} = d_j a_{ji}$
for all $i,j \in I$.
\end{itemize}
Thus $A = (a_{ij})_{i, j \in I}$ is a 
symmetrizable generalized Cartan matrix. Set $d_{ij} := -a_{ij}$ for short.
Fix also the additional data:
\begin{itemize}
\item a complex vector space
$\mathfrak{h}$;
\item linearly independent vectors
$\alpha_i \in \mathfrak{h}^*$ for each $i \in I$ called {\em simple roots};
\item linearly independent vectors 
$h_i \in \mathfrak{h}$ for each $i \in I$
such that $\langle h_i, \alpha_j \rangle =
a_{ij}$.
\end{itemize}
Let $\mathfrak{g}$ be
the associated Kac-Moody algebra 
with Chevalley generators $\{e_i, h_i, f_i\}_{i \in I}$ and 
Cartan subalgebra $\mathfrak{h}$.
Let $P := \{\lambda \in \mathfrak{h}^*\:|\:\langle h_i, \lambda\rangle \in \Z\text{
  for all }i \in I\}$ be its {\em weight lattice}.
In \cite{Rou}, Rouquier has
associated to $\mathfrak{g}$ a certain $2$-category
$\mathfrak{A}(\mathfrak{g})$, which we will denote instead by $\U(\mathfrak g)$.
It depends also on:
\begin{itemize}
\item a ground
ring $\k$;
\item 
units $t_{ij}\in\k^\times$ for $i,j \in I$
with $i \neq j$ such that
$t_{ij} = t_{ji}$ if $d_{ij}=0$;
\item
scalars $s_{ij}^{pq} \in \k$ for $i,j \in I$ and $0 
\leq p < d_{ij}, 0 \leq q < d_{ji}$
such that $s_{ij}^{pq} = s_{ji}^{qp}$.
\end{itemize}
(Often one assumes further that these parameters are {\em homogeneous} in
the sense that $s_{ij}^{pq} = 0$ unless $p d_i + q d_j = d_i d_{ij}$,
but we do not insist on that here.)
The following is the definition from \cite[$\S$4.1.3]{Rou} formulated in
diagrammatic terms.

\begin{definition}\label{def1}
The {\em Kac-Moody $2$-category}
$\U(\mathfrak{g})$
is the strict additive $\k$-linear $2$-category
with object set $P$,
generating 
$1$-morphisms 
$E_i 1_\lambda:\lambda \rightarrow \lambda+\alpha_i$ and
$F_i 1_\lambda:\lambda \rightarrow \lambda-\alpha_i$ for each $i \in I$ and
$\lambda \in P$, and generating
$2$-morphisms
$x:E_i 1_\lambda \rightarrow E_i 1_\lambda,
\tau:E_i E_j 1_\lambda \rightarrow E_j E_i 1_\lambda$,
$\eta:1_\lambda \rightarrow F_i E_i 1_\lambda$ and
$\eps:E_i F_i 1_\lambda \rightarrow 1_\lambda$,
subject to certain relations.
To record these, 
we adopt a diagrammatic formalism 
like in \cite{KL3}, representing the identity $2$-morphisms
of $E_i 1_\lambda$ and $F_i 1_\lambda$ by
${\scriptstyle\substack{\lambda+\alpha_i\\\phantom{-}}}\substack{{\color{darkred}{\displaystyle\uparrow}} \\
  {\scriptscriptstyle i}}{\scriptstyle\substack{\lambda\\\phantom{-}}}
$ and ${\scriptstyle\substack{\lambda-\alpha_i
\\\phantom{-}}}\substack{{\color{darkred}{\displaystyle\downarrow}} \\
  {\scriptscriptstyle
    i}}{\scriptstyle\substack{\lambda\\\phantom{-}}}$, respectively,
and the other generators by
\begin{align}\label{solid1}
x 
&= 
\mathord{
\begin{tikzpicture}[baseline = 0]
	\draw[->,thick,darkred] (0.08,-.3) to (0.08,.4);
      \node at (0.08,0.05) {$\color{darkred}\bullet$};
   \node at (0.08,-.4) {$\scriptstyle{i}$};
\end{tikzpicture}
}
{\scriptstyle\lambda}\:,
\qquad
\tau
= 
\mathord{
\begin{tikzpicture}[baseline = 0]
	\draw[->,thick,darkred] (0.28,-.3) to (-0.28,.4);
	\draw[->,thick,darkred] (-0.28,-.3) to (0.28,.4);
   \node at (-0.28,-.4) {$\scriptstyle{i}$};
   \node at (0.28,-.4) {$\scriptstyle{j}$};
   \node at (.4,.05) {$\scriptstyle{\lambda}$};
\end{tikzpicture}
}\:,
\qquad
\eta
= 
\mathord{
\begin{tikzpicture}[baseline = 0]
	\draw[<-,thick,darkred] (0.4,0.3) to[out=-90, in=0] (0.1,-0.1);
	\draw[-,thick,darkred] (0.1,-0.1) to[out = 180, in = -90] (-0.2,0.3);
    \node at (-0.2,.4) {$\scriptstyle{i}$};
  \node at (0.3,-0.15) {$\scriptstyle{\lambda}$};
\end{tikzpicture}
}\:,\qquad
\eps
= 
\mathord{
\begin{tikzpicture}[baseline = 0]
	\draw[<-,thick,darkred] (0.4,-0.1) to[out=90, in=0] (0.1,0.3);
	\draw[-,thick,darkred] (0.1,0.3) to[out = 180, in = 90] (-0.2,-0.1);
    \node at (-0.2,-.2) {$\scriptstyle{i}$};
  \node at (0.3,0.4) {$\scriptstyle{\lambda}$};
\end{tikzpicture}
}.
\end{align}
We stress that our diagrams are simply shorthands for
algebraic expressions built by horizontally and vertically composing
generators; they do not satisfy any  topological invariance other than the
``rectilinear isotopy'' implied by the
interchange law.
First, we have the {\em quiver Hecke relations}:
\begin{align}\label{qha}
\mathord{
\begin{tikzpicture}[baseline = 0]
	\draw[<-,thick,darkred] (0.25,.6) to (-0.25,-.2);
	\draw[->,thick,darkred] (0.25,-.2) to (-0.25,.6);
  \node at (-0.25,-.26) {$\scriptstyle{i}$};
   \node at (0.25,-.26) {$\scriptstyle{j}$};
  \node at (.3,.25) {$\scriptstyle{\lambda}$};
      \node at (-0.13,-0.02) {$\color{darkred}\bullet$};
\end{tikzpicture}
}
-
\mathord{
\begin{tikzpicture}[baseline = 0]
	\draw[<-,thick,darkred] (0.25,.6) to (-0.25,-.2);
	\draw[->,thick,darkred] (0.25,-.2) to (-0.25,.6);
  \node at (-0.25,-.26) {$\scriptstyle{i}$};
   \node at (0.25,-.26) {$\scriptstyle{j}$};
  \node at (.3,.25) {$\scriptstyle{\lambda}$};
      \node at (0.13,0.42) {$\color{darkred}\bullet$};
\end{tikzpicture}
}
&=
\mathord{
\begin{tikzpicture}[baseline = 0]
 	\draw[<-,thick,darkred] (0.25,.6) to (-0.25,-.2);
	\draw[->,thick,darkred] (0.25,-.2) to (-0.25,.6);
  \node at (-0.25,-.26) {$\scriptstyle{i}$};
   \node at (0.25,-.26) {$\scriptstyle{j}$};
  \node at (.3,.25) {$\scriptstyle{\lambda}$};
      \node at (-0.13,0.42) {$\color{darkred}\bullet$};
\end{tikzpicture}
}
-
\mathord{
\begin{tikzpicture}[baseline = 0]
	\draw[<-,thick,darkred] (0.25,.6) to (-0.25,-.2);
	\draw[->,thick,darkred] (0.25,-.2) to (-0.25,.6);
  \node at (-0.25,-.26) {$\scriptstyle{i}$};
   \node at (0.25,-.26) {$\scriptstyle{j}$};
  \node at (.3,.25) {$\scriptstyle{\lambda}$};
      \node at (0.13,-0.02) {$\color{darkred}\bullet$};
\end{tikzpicture}
}
=
\left\{
\begin{array}{ll}
\mathord{
\begin{tikzpicture}[baseline = 0]
 	\draw[->,thick,darkred] (0.08,-.3) to (0.08,.4);
	\draw[->,thick,darkred] (-0.28,-.3) to (-0.28,.4);
   \node at (-0.28,-.4) {$\scriptstyle{i}$};
   \node at (0.08,-.4) {$\scriptstyle{j}$};
 \node at (.28,.06) {$\scriptstyle{\lambda}$};
\end{tikzpicture}
}
&\text{if $i=j$,}\\
0&\text{otherwise,}\\
\end{array}
\right.
\end{align}
\begin{align}
\mathord{
\begin{tikzpicture}[baseline = 0]
	\draw[->,thick,darkred] (0.28,.4) to[out=90,in=-90] (-0.28,1.1);
	\draw[->,thick,darkred] (-0.28,.4) to[out=90,in=-90] (0.28,1.1);
	\draw[-,thick,darkred] (0.28,-.3) to[out=90,in=-90] (-0.28,.4);
	\draw[-,thick,darkred] (-0.28,-.3) to[out=90,in=-90] (0.28,.4);
  \node at (-0.28,-.4) {$\scriptstyle{i}$};
  \node at (0.28,-.4) {$\scriptstyle{j}$};
   \node at (.43,.4) {$\scriptstyle{\lambda}$};
\end{tikzpicture}
}
=
\left\{
\begin{array}{ll}
0&\text{if $i=j$,}\\
t_{ij}\mathord{
\begin{tikzpicture}[baseline = 0]
	\draw[->,thick,darkred] (0.08,-.3) to (0.08,.4);
	\draw[->,thick,darkred] (-0.28,-.3) to (-0.28,.4);
   \node at (-0.28,-.4) {$\scriptstyle{i}$};
   \node at (0.08,-.4) {$\scriptstyle{j}$};
   \node at (.3,.05) {$\scriptstyle{\lambda}$};
\end{tikzpicture}
}&\text{if $d_{ij}=0$,}\\
 t_{ij}
\mathord{
\begin{tikzpicture}[baseline = 0]
	\draw[->,thick,darkred] (0.08,-.3) to (0.08,.4);
	\draw[->,thick,darkred] (-0.28,-.3) to (-0.28,.4);
   \node at (-0.28,-.4) {$\scriptstyle{i}$};
   \node at (0.08,-.4) {$\scriptstyle{j}$};
   \node at (.3,-.05) {$\scriptstyle{\lambda}$};
      \node at (-0.28,0.05) {$\color{darkred}\bullet$};
      \node at (-0.5,0.2) {$\color{darkred}\scriptstyle{d_{ij}}$};
\end{tikzpicture}
}
+
t_{ji}
\mathord{
\begin{tikzpicture}[baseline = 0]
	\draw[->,thick,darkred] (0.08,-.3) to (0.08,.4);
	\draw[->,thick,darkred] (-0.28,-.3) to (-0.28,.4);
   \node at (-0.28,-.4) {$\scriptstyle{i}$};
   \node at (0.08,-.4) {$\scriptstyle{j}$};
   \node at (.3,-.05) {$\scriptstyle{\lambda}$};
     \node at (0.08,0.05) {$\color{darkred}\bullet$};
     \node at (0.32,0.2) {$\color{darkred}\scriptstyle{d_{ji}}$};
\end{tikzpicture}
}
+\!\! \displaystyle\sum_{\substack{0 \leq p < d_{ij}\\0 \leq q <
    d_{ji}}} \!\!\!\!\!s_{ij}^{pq}
\mathord{
\begin{tikzpicture}[baseline = 0]
	\draw[->,thick,darkred] (0.08,-.3) to (0.08,.4);
	\draw[->,thick,darkred] (-0.28,-.3) to (-0.28,.4);
   \node at (-0.28,-.4) {$\scriptstyle{i}$};
   \node at (0.08,-.4) {$\scriptstyle{j}$};
   \node at (.3,-.05) {$\scriptstyle{\lambda}$};
      \node at (-0.28,0.05) {$\color{darkred}\bullet$};
      \node at (0.08,0.05) {$\color{darkred}\bullet$};
      \node at (-0.43,0.2) {$\color{darkred}\scriptstyle{p}$};
      \node at (0.22,0.2) {$\color{darkred}\scriptstyle{q}$};
\end{tikzpicture}
}
&\text{otherwise,}\\
\end{array}
\right. \label{now}
\end{align}
\begin{align}
\mathord{
\begin{tikzpicture}[baseline = 0]
	\draw[<-,thick,darkred] (0.45,.8) to (-0.45,-.4);
	\draw[->,thick,darkred] (0.45,-.4) to (-0.45,.8);
        \draw[-,thick,darkred] (0,-.4) to[out=90,in=-90] (-.45,0.2);
        \draw[->,thick,darkred] (-0.45,0.2) to[out=90,in=-90] (0,0.8);
   \node at (-0.45,-.45) {$\scriptstyle{i}$};
   \node at (0,-.45) {$\scriptstyle{j}$};
  \node at (0.45,-.45) {$\scriptstyle{k}$};
   \node at (.5,-.1) {$\scriptstyle{\lambda}$};
\end{tikzpicture}
}
\!-
\!\!\!
\mathord{
\begin{tikzpicture}[baseline = 0]
	\draw[<-,thick,darkred] (0.45,.8) to (-0.45,-.4);
	\draw[->,thick,darkred] (0.45,-.4) to (-0.45,.8);
        \draw[-,thick,darkred] (0,-.4) to[out=90,in=-90] (.45,0.2);
        \draw[->,thick,darkred] (0.45,0.2) to[out=90,in=-90] (0,0.8);
   \node at (-0.45,-.45) {$\scriptstyle{i}$};
   \node at (0,-.45) {$\scriptstyle{j}$};
  \node at (0.45,-.45) {$\scriptstyle{k}$};
   \node at (.5,-.1) {$\scriptstyle{\lambda}$};
\end{tikzpicture}
}
&=
\left\{
\begin{array}{ll}
\displaystyle
\sum_{\substack{r,s \geq 0 \\ r+s=d_{ij}-1}}
\!\!\!
t_{ij}
\!
\mathord{
\begin{tikzpicture}[baseline = 0]
	\draw[->,thick,darkred] (0.44,-.3) to (0.44,.4);
	\draw[->,thick,darkred] (0.08,-.3) to (0.08,.4);
	\draw[->,thick,darkred] (-0.28,-.3) to (-0.28,.4);
   \node at (-0.28,-.4) {$\scriptstyle{i}$};
   \node at (0.08,-.4) {$\scriptstyle{j}$};
   \node at (0.44,-.4) {$\scriptstyle{k}$};
  \node at (.6,-.1) {$\scriptstyle{\lambda}$};
     \node at (-0.28,0.05) {$\color{darkred}\bullet$};
     \node at (0.44,0.05) {$\color{darkred}\bullet$};
      \node at (-0.43,0.2) {$\color{darkred}\scriptstyle{r}$};
      \node at (0.55,0.2) {$\color{darkred}\scriptstyle{s}$};
\end{tikzpicture}
}
+ 
\!\!\sum_{\substack{0 \leq p < d_{ij}\\0 \leq q <
    d_{ji}\\r,s \geq 0\\r+s=p-1}}
\!\!\!\!s_{ij}^{pq}
\mathord{
\begin{tikzpicture}[baseline = 0]
	\draw[->,thick,darkred] (0.44,-.3) to (0.44,.4);
	\draw[->,thick,darkred] (0.08,-.3) to (0.08,.4);
	\draw[->,thick,darkred] (-0.28,-.3) to (-0.28,.4);
   \node at (-0.28,-.4) {$\scriptstyle{i}$};
   \node at (0.08,-.4) {$\scriptstyle{j}$};
   \node at (0.44,-.4) {$\scriptstyle{k}$};
  \node at (.6,-.1) {$\scriptstyle{\lambda}$};
     \node at (-0.28,0.05) {$\color{darkred}\bullet$};
     \node at (0.44,0.05) {$\color{darkred}\bullet$};
      \node at (-0.43,0.2) {$\color{darkred}\scriptstyle{r}$};
     \node at (0.55,0.2) {$\color{darkred}\scriptstyle{s}$};
     \node at (0.08,0.05) {$\color{darkred}\bullet$};
      \node at (0.2,0.2) {$\color{darkred}\scriptstyle{q}$};
\end{tikzpicture}
}
&\text{if $i=k \neq j$,}\\
0&\text{otherwise.}
\end{array}\label{qhalast}
\right.\end{align}
(Note in the above relations that we represent powers of $x$ by decorating the dot with a
multiplicity.)
Next we have the {\em right adjunction relations}
\begin{align}\label{rightadj}
\mathord{
\begin{tikzpicture}[baseline = 0]
  \draw[->,thick,darkred] (0.3,0) to (0.3,.4);
	\draw[-,thick,darkred] (0.3,0) to[out=-90, in=0] (0.1,-0.4);
	\draw[-,thick,darkred] (0.1,-0.4) to[out = 180, in = -90] (-0.1,0);
	\draw[-,thick,darkred] (-0.1,0) to[out=90, in=0] (-0.3,0.4);
	\draw[-,thick,darkred] (-0.3,0.4) to[out = 180, in =90] (-0.5,0);
  \draw[-,thick,darkred] (-0.5,0) to (-0.5,-.4);
   \node at (-0.5,-.5) {$\scriptstyle{i}$};
   \node at (0.5,0) {$\scriptstyle{\lambda}$};
\end{tikzpicture}
}
&=
\mathord{\begin{tikzpicture}[baseline=0]
  \draw[->,thick,darkred] (0,-0.4) to (0,.4);
   \node at (0,-.5) {$\scriptstyle{i}$};
   \node at (0.2,0) {$\scriptstyle{\lambda}$};
\end{tikzpicture}
},\qquad
\mathord{
\begin{tikzpicture}[baseline = 0]
  \draw[->,thick,darkred] (0.3,0) to (0.3,-.4);
	\draw[-,thick,darkred] (0.3,0) to[out=90, in=0] (0.1,0.4);
	\draw[-,thick,darkred] (0.1,0.4) to[out = 180, in = 90] (-0.1,0);
	\draw[-,thick,darkred] (-0.1,0) to[out=-90, in=0] (-0.3,-0.4);
	\draw[-,thick,darkred] (-0.3,-0.4) to[out = 180, in =-90] (-0.5,0);
  \draw[-,thick,darkred] (-0.5,0) to (-0.5,.4);
   \node at (-0.5,.5) {$\scriptstyle{i}$};
   \node at (0.5,0) {$\scriptstyle{\lambda}$};
\end{tikzpicture}
}
=
\mathord{\begin{tikzpicture}[baseline=0]
  \draw[<-,thick,darkred] (0,-0.4) to (0,.4);
   \node at (0,.5) {$\scriptstyle{i}$};
   \node at (0.2,0) {$\scriptstyle{\lambda}$};
\end{tikzpicture}
},
\end{align}
which imply that $F_i 1_{\lambda+\alpha_i}$ is
the right dual of $E_i 1_\lambda$.
Finally there are some {\em inversion relations}. To formulate these, we first
introduce a new $2$-morphism
\begin{equation}\label{sigrel}
\sigma =  
\mathord{
\begin{tikzpicture}[baseline = 0]
	\draw[<-,thick,darkred] (0.28,-.3) to (-0.28,.4);
	\draw[->,thick,darkred] (-0.28,-.3) to (0.28,.4);
   \node at (-0.28,-.4) {$\scriptstyle{j}$};
   \node at (-0.28,.5) {$\scriptstyle{i}$};
   \node at (.4,.05) {$\scriptstyle{\lambda}$};
\end{tikzpicture}
}
:=
\mathord{
\begin{tikzpicture}[baseline = 0]
	\draw[->,thick,darkred] (0.3,-.5) to (-0.3,.5);
	\draw[-,thick,darkred] (-0.2,-.2) to (0.2,.3);
        \draw[-,thick,darkred] (0.2,.3) to[out=50,in=180] (0.5,.5);
        \draw[->,thick,darkred] (0.5,.5) to[out=0,in=90] (0.8,-.5);
        \draw[-,thick,darkred] (-0.2,-.2) to[out=230,in=0] (-0.5,-.5);
        \draw[-,thick,darkred] (-0.5,-.5) to[out=180,in=-90] (-0.8,.5);
  \node at (-0.8,.6) {$\scriptstyle{i}$};
   \node at (0.28,-.6) {$\scriptstyle{j}$};
   \node at (1.05,.05) {$\scriptstyle{\lambda}$};
\end{tikzpicture}
}
:E_j F_i 1_\lambda \rightarrow F_i E_j 1_\lambda.
\end{equation}
Then we require that the following $2$-morphisms are isomorphisms:
\begin{align}\label{inv1}
\mathord{
\begin{tikzpicture}[baseline = 0]
	\draw[<-,thick,darkred] (0.28,-.3) to (-0.28,.4);
	\draw[->,thick,darkred] (-0.28,-.3) to (0.28,.4);
   \node at (-0.28,-.4) {$\scriptstyle{j}$};
   \node at (-0.28,.5) {$\scriptstyle{i}$};
   \node at (.4,.05) {$\scriptstyle{\lambda}$};
\end{tikzpicture}
}
&:E_j F_i 1_\lambda \stackrel{\sim}{\rightarrow} F_i E_j 1_\lambda
&\text{if $i \neq j$,}\\
\label{inv2}
\mathord{
\begin{tikzpicture}[baseline = 0]
	\draw[<-,thick,darkred] (0.28,-.3) to (-0.28,.4);
	\draw[->,thick,darkred] (-0.28,-.3) to (0.28,.4);
   \node at (-0.28,-.4) {$\scriptstyle{i}$};
   \node at (-0.28,.5) {$\scriptstyle{i}$};
   \node at (.4,.05) {$\scriptstyle{\lambda}$};
\end{tikzpicture}
}
\oplus
\bigoplus_{n=0}^{\langle h_i,\lambda\rangle-1}
\mathord{
\begin{tikzpicture}[baseline = 0]
	\draw[<-,thick,darkred] (0.4,0) to[out=90, in=0] (0.1,0.4);
	\draw[-,thick,darkred] (0.1,0.4) to[out = 180, in = 90] (-0.2,0);
    \node at (-0.2,-.1) {$\scriptstyle{i}$};
  \node at (0.3,0.5) {$\scriptstyle{\lambda}$};
      \node at (-0.3,0.2) {$\color{darkred}\scriptstyle{n}$};
      \node at (-0.15,0.2) {$\color{darkred}\bullet$};
\end{tikzpicture}
}
&:
E_i F_i 1_\lambda\stackrel{\sim}{\rightarrow}
F_i E_i 1_\lambda \oplus 1_\lambda^{\oplus \langle h_i,\lambda\rangle}
&\text{if $\langle h_i,\lambda\rangle \geq
  0$},\\
\mathord{
\begin{tikzpicture}[baseline = 0]
	\draw[<-,thick,darkred] (0.28,-.3) to (-0.28,.4);
	\draw[->,thick,darkred] (-0.28,-.3) to (0.28,.4);
   \node at (-0.28,-.4) {$\scriptstyle{i}$};
   \node at (-0.28,.5) {$\scriptstyle{i}$};
   \node at (.4,.05) {$\scriptstyle{\lambda}$};
\end{tikzpicture}
}
\oplus
\bigoplus_{n=0}^{-\langle h_i,\lambda\rangle-1}
\mathord{
\begin{tikzpicture}[baseline = 0]
	\draw[<-,thick,darkred] (0.4,0.2) to[out=-90, in=0] (0.1,-.2);
	\draw[-,thick,darkred] (0.1,-.2) to[out = 180, in = -90] (-0.2,0.2);
    \node at (-0.2,.3) {$\scriptstyle{i}$};
  \node at (0.3,-0.25) {$\scriptstyle{\lambda}$};
      \node at (0.55,0) {$\color{darkred}\scriptstyle{n}$};
      \node at (0.38,0) {$\color{darkred}\bullet$};
\end{tikzpicture}
}
&
:E_i F_i 1_\lambda \oplus 
1_\lambda^{\oplus -\langle h_i,\lambda\rangle}
\stackrel{\sim}{\rightarrow}
 F_i E_i 1_\lambda&\text{if $\langle h_i,\lambda\rangle \leq
  0$}.\label{inv3}
\end{align}
(This means formally that there are some additional 
as yet unnamed generators which serve as
two-sided inverses to the $2$-morphisms in (\ref{inv1})--(\ref{inv3}).)
\end{definition}

Our main theorem identifies the $2$-category
$\U(\mathfrak{g})$ just defined 
with the Khovanov-Lauda $2$-category from \cite{KL3}. 
Actually Khovanov and Lauda worked just with the choice of parameters in which 
$t_{ij}=1$ and $s_{ij}^{pq}=0$ always. Subsequently, Cautis and Lauda
\cite{CL} generalized the definition to
incorporate more general choices of these parameters as above.
By the ``Khovanov-Lauda $2$-category'' we really mean the more
general version from \cite{CL}.

\vspace{2mm}
\noindent
{\bf Main Theorem.}
{\em Rouquier's Kac-Moody $2$-category $\U(\mathfrak{g})$ 
is isomorphic to the Khovanov-Lauda $2$-category.}
\vspace{2mm}

The proof is an elementary relation chase.
To explain the strategy, recall that the
Khovanov-Lauda $2$-category 
has the same objects and $1$-morphisms as $\U(\mathfrak{g})$.
Then there are generating $2$-morphisms 
represented by the same diagrams as
$x, \tau, \eta$ and $\eps$ above, plus
additional generating $2$-morphisms $x':F_i 1_\lambda \rightarrow F_i 1_\lambda$,
$\tau':F_i F_j 1_\lambda \rightarrow F_j F_i 1_\lambda$,
$\eta':1_\lambda \rightarrow E_i F_i 1_\lambda$ and $\eps':F_i E_i
1_\lambda \rightarrow 1_\lambda$
represented diagrammatically by
\begin{align}\label{solid2}
x'
&= 
\mathord{
\begin{tikzpicture}[baseline = 0]
	\draw[<-,thick,darkred] (0.08,-.3) to (0.08,.4);
      \node at (0.08,0.1) {$\color{darkred}\bullet$};
   \node at (0.08,.5) {$\scriptstyle{i}$};
\end{tikzpicture}
}
{\scriptstyle\lambda}\:,
\qquad
\tau'
= 
\mathord{
\begin{tikzpicture}[baseline = 0]
	\draw[<-,thick,darkred] (0.28,-.3) to (-0.28,.4);
	\draw[<-,thick,darkred] (-0.28,-.3) to (0.28,.4);
   \node at (-0.28,.5) {$\scriptstyle{j}$};
   \node at (0.28,.5) {$\scriptstyle{i}$};
   \node at (.4,.05) {$\scriptstyle{\lambda}$};
\end{tikzpicture}
}\:,
\qquad
\eta'
= 
\mathord{
\begin{tikzpicture}[baseline = 0]
	\draw[-,thick,darkred] (0.4,0.3) to[out=-90, in=0] (0.1,-0.1);
	\draw[->,thick,darkred] (0.1,-0.1) to[out = 180, in = -90] (-0.2,0.3);
    \node at (0.4,.4) {$\scriptstyle{i}$};
  \node at (0.3,-0.15) {$\scriptstyle{\lambda}$};
\end{tikzpicture}
}\:,\qquad
\eps'
= 
\mathord{
\begin{tikzpicture}[baseline = 0]
	\draw[-,thick,darkred] (0.4,-0.1) to[out=90, in=0] (0.1,0.3);
	\draw[->,thick,darkred] (0.1,0.3) to[out = 180, in = 90] (-0.2,-0.1);
    \node at (0.4,-.2) {$\scriptstyle{i}$};
  \node at (0.3,0.4) {$\scriptstyle{\lambda}$};
\end{tikzpicture}
}.
\end{align}
These satisfy further relations which
we will recall in more detail later in the introduction.
It is evident that all of the defining
relations of
$\U(\mathfrak{g})$ recorded above are satisfied in the Khovanov-Lauda
$2$-category.
Hence there is a strict $\k$-linear $2$-functor from
$\U(\mathfrak{g})$
to the Khovanov-Lauda $2$-category which is the identity on objects
and $1$-morphisms, and maps the generating $2$-morphisms $x,\tau,\eta$ and $\eps$ to
the corresponding $2$-morphisms from \cite{KL3,CL}.

To see that this functor is an isomorphism,
we construct a two-sided inverse.
In order to do this, we need to identify appropriate
$2$-morphisms
$x',\tau',\eta'$ and $\eps'$ in $\U(\mathfrak{g})$ that will be the images
of the additional generators (\ref{solid2}) under the inverse functor.
The definitions of $\eta'$ and $\eps'$ that follow
are essentially the same as Rouquier's
``candidates'' for second adjunction from \cite[$\S$4.1.4]{Rou},
except that we have renormalized by the sign $(-1)^{\langle h_i,\lambda\rangle+1}$ 
in order to be consistent with the conventions of \cite{KL3, CL}.
We will also define
a {\em leftward crossing}
\begin{equation}
\sigma' = 
\mathord{

}:
F_i E_j 1_\lambda \rightarrow E_j F_i 1_\lambda,
\end{equation}
which we have chosen to normalize
differently from the leftward crossing in \cite{CL}.

\begin{definition}
Define the downward dots and crossings $x'$ and $\tau'$
to be
the right mates of $x$ and $\tau$ (up to the
factor $t_{ij}^{-1}$ in the
latter case):
\begin{align}\label{xp}
x' &= 
\mathord{
%
}\:.\label{tp}
\end{align}
Then to define $\eta',\eps'$ and $\sigma'$,
we assume initially that
$\langle h_i,\lambda \rangle > 0$.
Thinking of (\ref{inv2}) as a column vector of
morphisms, its
inverse is a row vector. 
We define the $2$-morphisms $\sigma'$ and $\eta'$
so that $-\sigma'$ is the leftmost entry of this row vector and $\eta'$ is its
rightmost entry:
\begin{equation}\label{comfort}
-\mathord{
%
}
\bigg)^{-1}.
\end{equation}
 Instead, if
$\langle h_i,\lambda \rangle < 0$, the morphism (\ref{inv3}) is a row
vector and its inverse is a column vector.
We define $\sigma'$ and $\eps'$
so that $-\sigma'$ is the top entry of this column vector and $\eps'$
is its bottom one:
\begin{equation}\label{comfort2}
-\mathord{
%
}
\bigg)^{-1}.
\end{equation}
To complete the definitions of $\sigma',\eta'$ and $\eps'$, it remains
to set
\begin{align}
\label{hospitalR}
\mathord{
%
}
\bigg)^{-1}
&&\text{if $i \neq j$.}
\end{align}\end{definition}

Now to prove the Main Theorem we must show that all
of the defining relations for the Khovanov-Lauda $2$-category from
\cite{CL} are satisfied by the $2$-morphisms in 
$\U(\mathfrak{g})$ just introduced.
First we will show that
Lauda's {\em infinite Grassmannian relation} holds in
$\U(\mathfrak{g})$.
This asserts that,
as well as the dotted bubble $2$-morphisms
$\mathord{
%
}\in \End(1_\lambda)$
already defined for $r,s \geq 0$,
 there are  unique dotted bubble 
$2$-morphisms
for $r, s < 0$ such that
\begin{align}\label{awful1}
\mathord{
%
}
&= 0
\text{ for all $t > 0$}.\label{awful3}
\end{align}
Using the infinite Grassmannian relation, we deduce that the inverses of the
$2$-morphisms from (\ref{inv2}) 
and (\ref{inv3}) are
\begin{align}\label{yacht1}
\mathord{
-%
},
\end{align}
respectively.
Several other of the Khovanov-Lauda relations follow from
this assertion; see (\ref{pos})--(\ref{everything}) below.
After that,
we will show that the $2$-morphisms $\eta'$ and $\eps'$
define a unit and a counit making the
right dual $F_i 1_{\lambda+\alpha_i}$ of $E_i 1_\lambda$ also into its left dual:
\begin{equation}
\mathord{
%
}.\label{rigid}
\end{equation}
Consequently, $\U(\mathfrak{g})$ is {\em rigid}, i.e. all of its
$1$-morphisms admit both a left and a right dual.
Finally we will show that the other generating $2$-morphisms are {cyclic} (up to scalars):
\begin{equation}
\mathord{
%
}.\label{actress}
\end{equation}

\vspace{2mm}
\noindent
{\em Acknowledgements.}
The diagrams in this paper were created using Till Tantau's 
{\tt TikZ} package.

\section{Chevalley involution}

In this section we introduce a useful symmetry. First though we record some of the most basic
additional relations that hold in the $2$-category $\U(\mathfrak{g})$.

\begin{lemma}\label{skelling}
The following relations hold:
\begin{align}
\mathord{

}
=
\left\{
\begin{array}{ll}
\mathord{
\begin{tikzpicture}[baseline = 0]
 	\draw[-,thick,darkred] (-0.33,-.3) to [out=90,in=180] (-0.1,-0.02);
 	\draw[->,thick,darkred] (-0.1,-0.02) to [out=0,in=90] (0.13,-.3);
	\draw[-,thick,darkred] (-0.33,.4) to [out=-90,in=180] (-0.1,.12);
	\draw[->,thick,darkred] (-0.1,.12) to [out=0,in=-90](0.13,.4);
   \node at (-0.33,.5) {$\scriptstyle{j}$};
   \node at (-0.33,-.4) {$\scriptstyle{i}$};
 \node at (.33,.06) {$\scriptstyle{\lambda}$};
\end{tikzpicture}
}
&\text{if $i=j$,}\\
0&\text{otherwise,}\\
\end{array}%
\right.
\end{align}
\end{lemma}

\begin{proof}
The first two relations follow from the definition (\ref{xp}) of the
downward dot using
the adjunction relations (\ref{rightadj}).
The second two follow
similarly from the definition of the rightward crossing
(\ref{sigrel}).
For (\ref{rtcross}), attach a rightward cap to the top right
strand and
a rightward cup to the bottom left strand of (\ref{qha}), then use (\ref{first})
and the
definition (\ref{sigrel}).
Finally for (\ref{lurking}), attach a rightward cap to the top right strand and a rightward cup to the
bottom left strand in (\ref{qhalast}), then use
(\ref{rightpitchfork}) and the definition of the rightward crossing.
\end{proof}

Taking notation from \cite{CL}, we
define new parameters from
\begin{equation}\label{primed}
{'t}_{ij} := t_{ji}^{-1}, 
\qquad
{'s}_{ij}^{pq} := t_{ij}^{-1} t_{ji}^{-1}
s_{ji}^{qp}.
\end{equation}
The next lemma explains the significance of these scalars.

\begin{lemma}\label{opposite}
The following relations hold:
\begin{align}\label{qhadown}
\mathord{
\begin{tikzpicture}[baseline = 0]
	\draw[->,thick,darkred] (0.25,.6) to (-0.25,-.2);
	\draw[<-,thick,darkred] (0.25,-.2) to (-0.25,.6);
  \node at (-0.25,.7) {$\scriptstyle{i}$};
   \node at (0.25,.7) {$\scriptstyle{j}$};
  \node at (.3,.25) {$\scriptstyle{\lambda}$};
      \node at (0.13,0.42) {$\color{darkred}\bullet$};
\end{tikzpicture}
}
-
\mathord{
\begin{tikzpicture}[baseline = 0]
	\draw[->,thick,darkred] (0.25,.6) to (-0.25,-.2);
	\draw[<-,thick,darkred] (0.25,-.2) to (-0.25,.6);
  \node at (-0.25,.7) {$\scriptstyle{i}$};
   \node at (0.25,.7) {$\scriptstyle{j}$};
  \node at (.3,.25) {$\scriptstyle{\lambda}$};
      \node at (-0.13,-0.02) {$\color{darkred}\bullet$};
\end{tikzpicture}
}
=
\mathord{
\begin{tikzpicture}[baseline = 0]
 	\draw[->,thick,darkred] (0.25,.6) to (-0.25,-.2);
	\draw[<-,thick,darkred] (0.25,-.2) to (-0.25,.6);
  \node at (-0.25,.7) {$\scriptstyle{i}$};
   \node at (0.25,.7) {$\scriptstyle{j}$};
  \node at (.3,.25) {$\scriptstyle{\lambda}$};
     \node at (0.13,-0.02) {$\color{darkred}\bullet$};
 \end{tikzpicture}
}
-
\mathord{
\begin{tikzpicture}[baseline = 0]
	\draw[->,thick,darkred] (0.25,.6) to (-0.25,-.2);
	\draw[<-,thick,darkred] (0.25,-.2) to (-0.25,.6);
  \node at (-0.25,.7) {$\scriptstyle{i}$};
   \node at (0.25,.7) {$\scriptstyle{j}$};
  \node at (.3,.25) {$\scriptstyle{\lambda}$};
      \node at (-0.13,0.42) {$\color{darkred}\bullet$};
\end{tikzpicture}
}
&=
\left\{
\begin{array}{ll}
\mathord{
\begin{tikzpicture}[baseline = 0]
 	\draw[<-,thick,darkred] (0.08,-.3) to (0.08,.4);
	\draw[<-,thick,darkred] (-0.28,-.3) to (-0.28,.4);
   \node at (-0.28,.5) {$\scriptstyle{i}$};
   \node at (0.08,.5) {$\scriptstyle{j}$};
 \node at (.28,.06) {$\scriptstyle{\lambda}$};
\end{tikzpicture}
}
&\text{if $i=j$,}\\
0&\text{otherwise,}\\
\end{array}
\right.
\end{align}
\begin{align}
\mathord{
\begin{tikzpicture}[baseline = 0]
	\draw[-,thick,darkred] (0.28,.4) to[out=90,in=-90] (-0.28,1.1);
	\draw[-,thick,darkred] (-0.28,.4) to[out=90,in=-90] (0.28,1.1);
	\draw[<-,thick,darkred] (0.28,-.3) to[out=90,in=-90] (-0.28,.4);
	\draw[<-,thick,darkred] (-0.28,-.3) to[out=90,in=-90] (0.28,.4);
  \node at (-0.28,1.2) {$\scriptstyle{i}$};
  \node at (0.28,1.2) {$\scriptstyle{j}$};
   \node at (.43,.4) {$\scriptstyle{\lambda}$};
\end{tikzpicture}
}
&=
\left\{
\begin{array}{ll}
0&\text{if $i=j$,}\\
{'}t_{ij}\mathord{
\begin{tikzpicture}[baseline = 0]
	\draw[<-,thick,darkred] (0.08,-.3) to (0.08,.4);
	\draw[<-,thick,darkred] (-0.28,-.3) to (-0.28,.4);
   \node at (-0.28,.5) {$\scriptstyle{i}$};
   \node at (0.08,.5) {$\scriptstyle{j}$};
   \node at (.3,.05) {$\scriptstyle{\lambda}$};
\end{tikzpicture}
}&\text{if $d_{ij}=0$,}\\
{'t}_{ij}
\mathord{
\begin{tikzpicture}[baseline = 0]
	\draw[<-,thick,darkred] (0.08,-.3) to (0.08,.4);
	\draw[<-,thick,darkred] (-0.28,-.3) to (-0.28,.4);
   \node at (-0.28,.5) {$\scriptstyle{i}$};
   \node at (0.08,.5) {$\scriptstyle{j}$};
   \node at (.3,-.05) {$\scriptstyle{\lambda}$};
      \node at (-0.28,0.05) {$\color{darkred}\bullet$};
      \node at (-0.5,0.2) {$\color{darkred}\scriptstyle{d_{ij}}$};
\end{tikzpicture}
}
+
{'t}_{ji}
\mathord{
\begin{tikzpicture}[baseline = 0]
	\draw[<-,thick,darkred] (0.08,-.3) to (0.08,.4);
	\draw[<-,thick,darkred] (-0.28,-.3) to (-0.28,.4);
   \node at (-0.28,.5) {$\scriptstyle{i}$};
   \node at (0.08,.5) {$\scriptstyle{j}$};
   \node at (.3,-.05) {$\scriptstyle{\lambda}$};
     \node at (0.08,0.05) {$\color{darkred}\bullet$};
     \node at (0.32,0.2) {$\color{darkred}\scriptstyle{d_{ji}}$};
\end{tikzpicture}
}
+\!\! \displaystyle\sum_{\substack{0 \leq p < d_{ij}\\0 \leq q <
    d_{ji}}} \!\!\!\!\!{'s}_{ij}^{pq}
\mathord{
\begin{tikzpicture}[baseline = 0]
	\draw[<-,thick,darkred] (0.08,-.3) to (0.08,.4);
	\draw[<-,thick,darkred] (-0.28,-.3) to (-0.28,.4);
   \node at (-0.28,.5) {$\scriptstyle{i}$};
   \node at (0.08,.5) {$\scriptstyle{j}$};
   \node at (.3,-.05) {$\scriptstyle{\lambda}$};
      \node at (-0.28,0.05) {$\color{darkred}\bullet$};
      \node at (0.08,0.05) {$\color{darkred}\bullet$};
      \node at (-0.43,0.2) {$\color{darkred}\scriptstyle{p}$};
      \node at (0.22,0.2) {$\color{darkred}\scriptstyle{q}$};
\end{tikzpicture}
}
&\text{otherwise,}\\
\end{array}
\right.
\end{align}\begin{align}
\mathord{
\begin{tikzpicture}[baseline = 0]
	\draw[->,thick,darkred] (0.45,.8) to (-0.45,-.4);
	\draw[<-,thick,darkred] (0.45,-.4) to (-0.45,.8);
        \draw[<-,thick,darkred] (0,-.4) to[out=90,in=-90] (.45,0.2);
        \draw[-,thick,darkred] (0.45,0.2) to[out=90,in=-90] (0,0.8);
   \node at (-0.45,.9) {$\scriptstyle{i}$};
   \node at (0,.9) {$\scriptstyle{j}$};
  \node at (0.45,.9) {$\scriptstyle{k}$};
   \node at (.5,-.1) {$\scriptstyle{\lambda}$};
\end{tikzpicture}
}
\!-
\!\!\!
\mathord{
\begin{tikzpicture}[baseline = 0]
	\draw[->,thick,darkred] (0.45,.8) to (-0.45,-.4);
	\draw[<-,thick,darkred] (0.45,-.4) to (-0.45,.8);
        \draw[<-,thick,darkred] (0,-.4) to[out=90,in=-90] (-.45,0.2);
        \draw[-,thick,darkred] (-0.45,0.2) to[out=90,in=-90] (0,0.8);
   \node at (-0.45,.9) {$\scriptstyle{i}$};
   \node at (0,.9) {$\scriptstyle{j}$};
  \node at (0.45,.9) {$\scriptstyle{k}$};
   \node at (.5,-.1) {$\scriptstyle{\lambda}$};
\end{tikzpicture}
}
&=
\left\{
\begin{array}{ll}
\!\!\displaystyle
\sum_{\substack{r,s \geq 0 \\ r+s=d_{ij}-1}}
\!\!\!
{'t}_{ij}
\!
\mathord{
\begin{tikzpicture}[baseline = 0]
	\draw[<-,thick,darkred] (0.44,-.3) to (0.44,.4);
	\draw[<-,thick,darkred] (0.08,-.3) to (0.08,.4);
	\draw[<-,thick,darkred] (-0.28,-.3) to (-0.28,.4);
   \node at (-0.28,.5) {$\scriptstyle{i}$};
   \node at (0.08,.5) {$\scriptstyle{j}$};
   \node at (0.44,.5) {$\scriptstyle{k}$};
  \node at (.6,-.1) {$\scriptstyle{\lambda}$};
     \node at (-0.28,0.05) {$\color{darkred}\bullet$};
     \node at (0.44,0.05) {$\color{darkred}\bullet$};
      \node at (-0.43,0.2) {$\color{darkred}\scriptstyle{r}$};
      \node at (0.55,0.2) {$\color{darkred}\scriptstyle{s}$};
\end{tikzpicture}
}
+ 
\!\!\!\sum_{\substack{0 \leq p < d_{ij}\\0 \leq q <
    d_{ji}\\r,s \geq 0\\r+s=p-1}}
\!\!\!\!{'s}_{ij}^{pq}
\mathord{
\begin{tikzpicture}[baseline = 0]
	\draw[<-,thick,darkred] (0.44,-.3) to (0.44,.4);
	\draw[<-,thick,darkred] (0.08,-.3) to (0.08,.4);
	\draw[<-,thick,darkred] (-0.28,-.3) to (-0.28,.4);
   \node at (-0.28,.5) {$\scriptstyle{i}$};
   \node at (0.08,.5) {$\scriptstyle{j}$};
   \node at (0.44,.5) {$\scriptstyle{k}$};
  \node at (.6,-.1) {$\scriptstyle{\lambda}$};
     \node at (-0.28,0.05) {$\color{darkred}\bullet$};
     \node at (0.44,0.05) {$\color{darkred}\bullet$};
      \node at (-0.43,0.2) {$\color{darkred}\scriptstyle{r}$};
     \node at (0.55,0.2) {$\color{darkred}\scriptstyle{s}$};
     \node at (0.08,0.05) {$\color{darkred}\bullet$};
      \node at (0.2,0.2) {$\color{darkred}\scriptstyle{q}$};
\end{tikzpicture}
}\!\!
&\text{if $i=k \neq j$,}\\
\!\!0&\text{otherwise.}
\end{array}
\right.\end{align}
\end{lemma}

\begin{proof}
Put rightward caps on the top and rightward cups on the bottom of the
relations (\ref{qha})--(\ref{qhalast}), then use (\ref{rightadj}),
the definitions (\ref{sigrel}), (\ref{xp}), (\ref{tp}), and (\ref{first})--(\ref{rightpitchfork}).
\end{proof}

For any strict $\k$-linear $2$-category $\mathcal C$, we write
$\mathcal C^{\operatorname{opp}}$ for the $2$-category with the same
objects as $\mathcal C$
but with morphism categories defined from
$\mathcal{H}om_{\mathcal C^{\operatorname{opp}}}(\lambda,\mu) :=
\mathcal{H}om_{\mathcal C}(\lambda,\mu)^{\operatorname{opp}}$, 
so the vertical composition in $\mathcal C^{\operatorname{opp}}$ is the
opposite of the one in $\mathcal C$, while
the horizontal composition in $\mathcal C^{\operatorname{opp}}$ is
the same as in $\mathcal C$.

\begin{theorem}\label{opiso}
Let ${'\U}(\mathfrak{g})$ be the Kac-Moody $2$-category
defined as in Definition~\ref{def1} but using the primed 
parameters from (\ref{primed}) in place of $t_{ij}$ and $s_{ij}^{pq}$.
Then there is an isomorphism of strict $\k$-linear $2$-categories
$$
\T:{'\U}(\mathfrak{g}) \stackrel{\sim}{\rightarrow}
\U(\mathfrak{g})^{\operatorname{opp}}
$$
defined on objects by $\T(\lambda) := -\lambda$, on generating
$1$-morphisms
by $\T(E_i 1_\lambda) := F_i 1_{-\lambda}$ and
$\T(F_i 1_\lambda) := E_i 1_{-\lambda}$, and on generating
$2$-morphisms by:
\begin{align*}
\mathord{
\begin{tikzpicture}[baseline = 0]
	\draw[->,thick,darkred] (0.08,-.3) to (0.08,.4);
      \node at (0.08,0.05) {$\color{darkred}\bullet$};
   \node at (0.08,-.4) {$\scriptstyle{i}$};
\end{tikzpicture}
}
{\scriptstyle\lambda}\:\,\mapsto
\mathord{
\begin{tikzpicture}[baseline = 0]
	\draw[<-,thick,darkred] (0.08,-.3) to (0.08,.4);
      \node at (0.08,0.1) {$\color{darkred}\bullet$};
     \node at (0.08,.5) {$\scriptstyle{i}$};
\end{tikzpicture}
}
{\scriptstyle-\lambda}\:,
\qquad
\mathord{
\begin{tikzpicture}[baseline = 0]
	\draw[->,thick,darkred] (0.28,-.3) to (-0.28,.4);
	\draw[->,thick,darkred] (-0.28,-.3) to (0.28,.4);
   \node at (-0.28,-.4) {$\scriptstyle{i}$};
   \node at (0.28,-.4) {$\scriptstyle{j}$};
   \node at (.4,.05) {$\scriptstyle{\lambda}$};
\end{tikzpicture}
}&\mapsto
-\mathord{
\begin{tikzpicture}[baseline = 0]
	\draw[<-,thick,darkred] (0.28,-.3) to (-0.28,.4);
	\draw[<-,thick,darkred] (-0.28,-.3) to (0.28,.4);
   \node at (-0.28,.5) {$\scriptstyle{i}$};
   \node at (0.28,.5) {$\scriptstyle{j}$};
   \node at (.4,.05) {$\scriptstyle{-\lambda}$};
\end{tikzpicture}
},
\qquad
\mathord{
\begin{tikzpicture}[baseline = 0]
	\draw[<-,thick,darkred] (0.4,0.3) to[out=-90, in=0] (0.1,-0.1);
	\draw[-,thick,darkred] (0.1,-0.1) to[out = 180, in = -90] (-0.2,0.3);
    \node at (-0.2,.4) {$\scriptstyle{i}$};
  \node at (0.3,-0.15) {$\scriptstyle{\lambda}$};
\end{tikzpicture}
}
\mapsto
\!\mathord{
\begin{tikzpicture}[baseline = 0]
	\draw[<-,thick,darkred] (0.4,-0.1) to[out=90, in=0] (0.1,0.3);
	\draw[-,thick,darkred] (0.1,0.3) to[out = 180, in = 90] (-0.2,-0.1);
    \node at (-0.2,-.2) {$\scriptstyle{i}$};
  \node at (0.3,0.4) {$\scriptstyle{-\lambda}$};
\end{tikzpicture}
},
\qquad
\mathord{
\begin{tikzpicture}[baseline = 0]
	\draw[<-,thick,darkred] (0.4,-0.1) to[out=90, in=0] (0.1,0.3);
	\draw[-,thick,darkred] (0.1,0.3) to[out = 180, in = 90] (-0.2,-0.1);
    \node at (-0.2,-.2) {$\scriptstyle{i}$};
  \node at (0.3,0.4) {$\scriptstyle{\lambda}$};
\end{tikzpicture}
}\mapsto
\!\mathord{
\begin{tikzpicture}[baseline = 0]
	\draw[<-,thick,darkred] (0.4,0.3) to[out=-90, in=0] (0.1,-0.1);
	\draw[-,thick,darkred] (0.1,-0.1) to[out = 180, in = -90] (-0.2,0.3);
    \node at (-0.2,.4) {$\scriptstyle{i}$};
  \node at (0.32,-0.18) {$\scriptstyle{-\lambda}$};
\end{tikzpicture}
}.\\\intertext{The effect of $\T$ on the other named $2$-morphisms in
${'\U}(\mathfrak{g})$ is as follows:}
\mathord{
\begin{tikzpicture}[baseline = 0]
	\draw[<-,thick,darkred] (0.08,-.3) to (0.08,.4);
      \node at (0.08,0.1) {$\color{darkred}\bullet$};
     \node at (0.08,.5) {$\scriptstyle{i}$};
\end{tikzpicture}
}
{\scriptstyle\lambda}
\:\,\mapsto
\mathord{
\begin{tikzpicture}[baseline = 0]
	\draw[->,thick,darkred] (0.08,-.3) to (0.08,.4);
      \node at (0.08,0.05) {$\color{darkred}\bullet$};
   \node at (0.08,-.4) {$\scriptstyle{i}$};
\end{tikzpicture}
}
{\scriptstyle-\lambda},
\qquad
\mathord{
\begin{tikzpicture}[baseline = 0]
	\draw[<-,thick,darkred] (0.28,-.3) to (-0.28,.4);
	\draw[<-,thick,darkred] (-0.28,-.3) to (0.28,.4);
   \node at (-0.28,.5) {$\scriptstyle{j}$};
   \node at (0.28,.5) {$\scriptstyle{i}$};
   \node at (.4,.05) {$\scriptstyle{\lambda}$};
\end{tikzpicture}}&\mapsto
-
\mathord{
\begin{tikzpicture}[baseline = 0]
	\draw[->,thick,darkred] (0.28,-.3) to (-0.28,.4);
	\draw[->,thick,darkred] (-0.28,-.3) to (0.28,.4);
   \node at (-0.28,-.4) {$\scriptstyle{j}$};
   \node at (0.28,-.4) {$\scriptstyle{i}$};
   \node at (.4,.05) {$\scriptstyle{-\lambda}$};
\end{tikzpicture}
},\qquad
\mathord{
\begin{tikzpicture}[baseline = 0]
	\draw[-,thick,darkred] (0.4,0.3) to[out=-90, in=0] (0.1,-0.1);
	\draw[->,thick,darkred] (0.1,-0.1) to[out = 180, in = -90] (-0.2,0.3);
    \node at (0.4,.4) {$\scriptstyle{i}$};
  \node at (0.3,-0.15) {$\scriptstyle{\lambda}$};
\end{tikzpicture}
}
\mapsto
\:\mathord{
\begin{tikzpicture}[baseline = 0]
	\draw[-,thick,darkred] (0.4,-0.1) to[out=90, in=0] (0.1,0.3);
	\draw[->,thick,darkred] (0.1,0.3) to[out = 180, in = 90] (-0.2,-0.1);
    \node at (0.4,-.2) {$\scriptstyle{i}$};
  \node at (0.3,0.4) {$\scriptstyle{-\lambda}$};
\end{tikzpicture}
},
\qquad
\mathord{
\begin{tikzpicture}[baseline = 0]
	\draw[-,thick,darkred] (0.4,-0.1) to[out=90, in=0] (0.1,0.3);
	\draw[->,thick,darkred] (0.1,0.3) to[out = 180, in = 90] (-0.2,-0.1);
    \node at (0.4,-.2) {$\scriptstyle{i}$};
  \node at (0.3,0.4) {$\scriptstyle{\lambda}$};
\end{tikzpicture}
}\mapsto
\:\mathord{
\begin{tikzpicture}[baseline = 0]
	\draw[-,thick,darkred] (0.4,0.3) to[out=-90, in=0] (0.1,-0.1);
	\draw[->,thick,darkred] (0.1,-0.1) to[out = 180, in = -90] (-0.2,0.3);
    \node at (0.4,.4) {$\scriptstyle{i}$};
  \node at (0.32,-0.18) {$\scriptstyle{-\lambda}$};
\end{tikzpicture}
},\\
\mathord{
\begin{tikzpicture}[baseline = 0]
	\draw[<-,thick,darkred] (0.28,-.3) to (-0.28,.4);
	\draw[->,thick,darkred] (-0.28,-.3) to (0.28,.4);
   \node at (-0.28,-.4) {$\scriptstyle{i}$};
   \node at (-0.28,.5) {$\scriptstyle{j}$};
   \node at (.4,.05) {$\scriptstyle{\lambda}$};
\end{tikzpicture}
}
&\mapsto
-t_{ij}^{-1}
\mathord{
\begin{tikzpicture}[baseline = 0]
	\draw[<-,thick,darkred] (0.28,-.3) to (-0.28,.4);
	\draw[->,thick,darkred] (-0.28,-.3) to (0.28,.4);
   \node at (-0.28,-.4) {$\scriptstyle{j}$};
   \node at (-0.28,.5) {$\scriptstyle{i}$};
   \node at (.4,.05) {$\scriptstyle{-\lambda}$};
\end{tikzpicture}
},
\qquad
\mathord{
\begin{tikzpicture}[baseline = 0]
	\draw[->,thick,darkred] (0.28,-.3) to (-0.28,.4);
	\draw[<-,thick,darkred] (-0.28,-.3) to (0.28,.4);
   \node at (0.28,-.4) {$\scriptstyle{i}$};
   \node at (0.28,.5) {$\scriptstyle{j}$};
   \node at (.4,.05) {$\scriptstyle{\lambda}$};
\end{tikzpicture}
}
\mapsto
-
t_{ij}\:
\mathord{
\begin{tikzpicture}[baseline = 0]
	\draw[->,thick,darkred] (0.28,-.3) to (-0.28,.4);
	\draw[<-,thick,darkred] (-0.28,-.3) to (0.28,.4);
   \node at (0.28,-.4) {$\scriptstyle{j}$};
   \node at (0.28,.5) {$\scriptstyle{i}$};
   \node at (.4,.05) {$\scriptstyle{-\lambda}$};
\end{tikzpicture}
}.
\end{align*}
Note in particular that $\T^2 = \operatorname{id}$.
\end{theorem}

\begin{proof}
To see that $\T$ is well defined we need to verify that the 
images
under $\T$ of the relations
(\ref{qha})--(\ref{rightadj}) and (\ref{inv1})--(\ref{inv3})
with primed parameters hold in $\U(\mathfrak{g})^{\operatorname{opp}}$. For the first
three, this follows from Lemma~\ref{opposite}, while (\ref{rightadj})
is clear.
For the remaining ones, we first note that
$\mathord{
\begin{tikzpicture}[baseline = 0]
	\draw[<-,thick,darkred] (0.28,-.3) to (-0.28,.4);
	\draw[->,thick,darkred] (-0.28,-.3) to (0.28,.4);
   \node at (-0.28,-.4) {$\scriptstyle{i}$};
   \node at (-0.28,.5) {$\scriptstyle{i}$};
   \node at (.4,.05) {$\scriptstyle{\lambda}$};
\end{tikzpicture}
}\mapsto
-\mathord{
\begin{tikzpicture}[baseline = 0]
	\draw[<-,thick,darkred] (0.28,-.3) to (-0.28,.4);
	\draw[->,thick,darkred] (-0.28,-.3) to (0.28,.4);
   \node at (-0.28,-.4) {$\scriptstyle{i}$};
   \node at (-0.28,.5) {$\scriptstyle{i}$};
   \node at (.4,.05) {$\scriptstyle{-\lambda}$};
\end{tikzpicture}
}$
by the definition (\ref{sigrel}).
Then for example for (\ref{inv2}), we must show for $\langle
h_i,\lambda \rangle \geq 0$ that
$\mathord{
\begin{tikzpicture}[baseline = 0]
	\draw[<-,thick,darkred] (0.28,-.3) to (-0.28,.4);
	\draw[->,thick,darkred] (-0.28,-.3) to (0.28,.4);
   \node at (-0.28,-.4) {$\scriptstyle{i}$};
   \node at (-0.28,.5) {$\scriptstyle{i}$};
   \node at (.4,.05) {$\scriptstyle{-\lambda}$};
\end{tikzpicture}
}
\oplus
\displaystyle
\bigoplus_{n=0}^{\langle h_i,\lambda\rangle-1}
\mathord{
\begin{tikzpicture}[baseline = 0]
	\draw[<-,thick,darkred] (0.4,0.2) to[out=-90, in=0] (0.1,-.2);
	\draw[-,thick,darkred] (0.1,-.2) to[out = 180, in = -90] (-0.2,0.2);
    \node at (-0.2,.3) {$\scriptstyle{i}$};
  \node at (0.32,-0.28) {$\scriptstyle{-\lambda}$};
      \node at (0.55,0) {$\color{darkred}\scriptstyle{n}$};
      \node at (-0.15,0) {$\color{darkred}\bullet$};
\end{tikzpicture}
}$
is invertible in $\U(\mathfrak{g})$. This follows by composing (\ref{inv3})
with $\operatorname{diag}(-1,1,\dots,1)$, 
using also (\ref{first}).
The rest of the theorem is a routine check from the definitions (\ref{xp})--(\ref{hospitalL}).
\end{proof}
 
We will often appeal to Theorem~\ref{opiso} to establish mirror images of
relations in a horizontal axis. For example, applying it to
(\ref{rightpitchfork}), we obtain the following relation (which could
also be deduced directly from the definition of the
downward crossing):

\begin{corollary}
The following relations hold:
\begin{align}
\mathord{

}.\label{rightpitchfork2}
\end{align}
\end{corollary}

\section{The infinite Grassmannian relation}

The goal in this section is to show that 
the infinite Grassmannian relation holds
in 
$\U(\mathfrak{g})$. 
We begin by introducing notation for the other entries of the row vector that
is the two-sided inverse of (\ref{inv2})
and of the column vector that is the two-sided inverse of (\ref{inv3}):
let
\begin{align}\label{spades}
-
\mathord{
%
}
\bigg)^{-1}
&&
\text{if $\langle h_i,\lambda \rangle \leq 0$.}\label{clubs}
\end{align}
We will give a more explicit description of these $2$-morphisms in
(\ref{invA})--(\ref{invB})
below.
Note right away 
comparing the present definitions with
(\ref{comfort})--(\ref{comfort2})
that
\begin{equation}
\mathord{
%
}
\text{ if $\langle h_i,\lambda \rangle < 0$.}\label{die}
\end{equation}
Also for all admissible values of $n$ we have that
\begin{equation}\label{poirot}
\T\bigg(\:\:\:\mathord{
%
},
\end{equation}
as follows on applying the isomorphism $\T$ from Theorem~\ref{opiso} to the
definitions
(\ref{spades})--(\ref{clubs})
and using (\ref{first}).

\begin{lemma}\label{start}
The following relations hold:
\begin{align}\label{posneg}
\mathord{
%
} = \delta_{n,-\langle h_i,\lambda\rangle-1} 1_{1_\lambda}
\text{ all assuming $0 \leq n < -\langle h_i,\lambda\rangle$}.\label{startd}
\end{align}
\end{lemma}

\begin{proof}
This follows from (\ref{spades})--(\ref{die}).
\end{proof}

The dotted bubbles
$\mathord{
%
}$ 
define endomorphisms of $1_\lambda$
for $r,s \geq 0$.
We also give meaning to negatively dotted bubbles by making the
following definitions for $r, s < 0$:
\begin{align}
\mathord{

}
&:=
\left\{
\begin{array}{cl}
-
\mathord{
\begin{tikzpicture}[baseline = 0]
  \draw[-,thick,darkred] (0.2,0.1) to[out=90,in=0] (0,.3);
  \draw[-,thick,darkred] (0,0.3) to[out=180,in=90] (-.2,0.1);
\draw[->,thick,darkred] (-.2,0.1) to[out=-90,in=180] (0,-0.1);
  \draw[-,thick,darkred] (0,-0.1) to[out=0,in=-90] (0.2,0.1);
 \node at (-.3,.3) {$\scriptstyle{i}$};
   \node at (-.4,0) {$\scriptstyle{\lambda}$};
   \node at (0.2,0.1) {$\color{darkred}\bullet$};
   \node at (0,0.3) {$\color{darkred}\scriptstyle{\clubsuit}$};
   \node at (0.72,0.1) {$\color{darkred}\scriptstyle{-\langle h_i,\lambda\rangle}$};
   \node at (0,0.5) {$\color{darkred}\scriptstyle{-r-1}$};
\end{tikzpicture}
}&\text{if $r > \langle h_i,\lambda\rangle-1$,}\\
1_{1_\lambda}&\text{if $r=\langle h_i, \lambda\rangle-1$},\\
0&\text{if $r< \langle h_i,\lambda\rangle-1$},
\end{array}%
\right.
\label{ig2}
\end{align}
Note by Theorem~\ref{opiso}, (\ref{first}) and (\ref{poirot})
that
\begin{equation}\label{doublebubble}
\T\left(\mathord{
%
},
\end{equation}
for all $r,s \in \Z$.
Given (\ref{startb})--(\ref{ig2}),
the following theorem implies the infinite
Grassmannian relation as formulated in formulae
(\ref{awful1})--(\ref{awful3}) in the introduction.

\begin{theorem}\label{IG}
The following holds for all $t > 0$:
\begin{align}\label{ig3}
\sum_{\substack{r,s \in \Z \\ r+s=t-2}}
\mathord{
%
}
&= 0.
\end{align}
\end{theorem}

\begin{proof}
We prove this under the assumption that $\langle h_i, \lambda \rangle
\geq 0$;
the result when $\langle h_i,\lambda \rangle \leq 0$ then follows using
Theorem~\ref{opiso} and (\ref{doublebubble}).
We have that
\begin{align*}
\sum_{\substack{r,s \in \Z \\ r+s=t-2}}
&\mathord{
%
}.
\end{align*}
It just remains to show that the final expression here is zero.
When $\langle h_i,\lambda \rangle > 0$ this follows by
(\ref{startb}) and (\ref{ig1}).
Finally if $\langle h_i,\lambda \rangle = 0$ then
$$
\mathord{
%
}
\stackrel{(\ref{posneg})}{=} 0.
$$
\end{proof}

\begin{corollary}\label{advances}
The following relations hold:
\begin{align}\label{invA}
\mathord{
%
}
&&\text{if $0 \leq n < -\langle h_i,\lambda\rangle$.}\label{invB}
\end{align}
\end{corollary}

\begin{proof}
We explain the proof of (\ref{invA}); the proof of (\ref{invB}) is
entirely similar.
Remembering the definition (\ref{spades}), 
it suffices to show that the vertical composition consisting of (\ref{inv2})
on top of (\ref{yacht1})
is equal to the identity.
Using (\ref{posneg})--(\ref{startb}),
this reduces to checking that
\begin{align}
\sum_{r \geq 0}
\mathord{
%
}
&=\delta_{m,n} 1_{1_\lambda}
&&\text{if $0 \leq m,n < \langle h_i,\lambda\rangle$.}
\label{better2}\end{align}
For (\ref{better1}), each term in the summation is zero:
if $r \geq 
\langle h_i,\lambda\rangle$ the counterclockwise dotted bubble is zero
by (\ref{ig2}); if $0 \leq r < \langle h_i,\lambda \rangle$ we have that
$$
\mathord{
%
}
\stackrel{(\ref{startb})}{=} 0.
$$
To prove (\ref{better2}), note 
by (\ref{startb}) and (\ref{ig2}) that
in order for 
$\mathord{
%
}$ 
to be non-zero we must have
$m+r \geq \langle h_i,\lambda\rangle-1$ and $-n-r-2 \geq -\langle h_i,\lambda\rangle-1$.
Adding these inequalities implies that $m \geq n$.
Moreover if $m=n$ the only non-zero term in the summation is the term
with
$r = \langle h_i,\lambda\rangle-1-m$, which equals $1_{1_\lambda}$ by
(\ref{startb}) and (\ref{ig2})
again.
Finally if $m > n$ the left hand side of (\ref{better2}) can be
rewritten as
$$
\sum_{\substack{r,s \in \Z \\ r+s=m-n-2}}
\mathord{
%
},
$$
which is zero by (\ref{ig3}).
\end{proof}

On substituting (\ref{invA})--(\ref{invB}) into the definitions
(\ref{spades})--(\ref{clubs}), this
establishes the assertions about the $2$-morphisms (\ref{yacht1})--(\ref{yacht2}) made in the introduction.

\begin{corollary}
The following relations hold:
\begin{align}\label{pos}
\mathord{
%
}.
\end{align}
\end{corollary}

\begin{proof}
Substitute (\ref{invA})--(\ref{invB}) into (\ref{posneg}).
\end{proof}

\begin{corollary}\label{curry}
The following relations hold:
\begin{align}\label{everything}
\mathord{
%
}.
\end{align}
\end{corollary}

\begin{proof}
In view of Theorem~\ref{opiso}, (\ref{first}) and (\ref{doublebubble}), it suffices to prove the left hand relation.
We are done already by (\ref{startd}) if $\langle
h_i,\lambda \rangle < 0$. 
If $\langle h_i,\lambda \rangle \geq 0$ then:
\begin{align*}
\mathord{
%
}.
\end{align*}
\end{proof}

\section{Left adjunction relations}

In this section we show that the leftward cup and cap 
satisfy adjunction relations.

\begin{lemma}
The following relations hold:
\begin{align}\label{dog1}
\mathord{
%
}
&&\text{if $0 \leq n \leq \langle h_i,\lambda\rangle$.}\label{dog1b}
\end{align}
\end{lemma}

\begin{proof}
Proceed by induction on $n$.
For the base case, 
convert the upward crossings to rightward ones
using (\ref{rightadj})--(\ref{sigrel}), apply (\ref{everything}) and (\ref{ig1})--(\ref{ig2}), then invoke
(\ref{rightadj}).
For the induction step, pull a dot past the crossing using (\ref{first})
and (\ref{qha}),
then use (\ref{startb})--(\ref{startd}) and the induction hypothesis.
\end{proof}

\begin{lemma}
Then the following relations hold:
\begin{align}\label{one}
\mathord{

}&&\text{if $\langle h_i,\lambda\rangle > -2$.}\label{two}
\end{align}
\end{lemma}

\begin{proof}
Let $h := \langle h_i,\lambda\rangle$ for short.
First we prove (\ref{one}), so $h < 0$.
We claim that
\begin{equation}
-\!\!\!
\mathord{
%
}.\label{quarry}
\end{equation}
To establish the claim, we vertically compose on the bottom with the
isomorphism
$\mathord{
%
}$
arising from (\ref{inv3})
to reduce to showing equivalently that
\begin{align}
-\mathord{
%
}
\text{ for $0 \leq n \leq -h-1$.}\label{claude2}
\end{align}
Here is the verification of (\ref{claude1}):
\begin{align*}
-\mathord{
%
}.
\end{align*}
For (\ref{claude2}), we note by (\ref{startd}) and (\ref{rightadj})
that the right hand side is equal to $\substack{{\color{darkred}{\displaystyle\uparrow}} \\
  {\scriptscriptstyle i}}{\scriptstyle\substack{\lambda\\\phantom{-}}}
$ if $n = -h-1 > 0$, and it is
zero otherwise.
Now we simplify the left hand side:
\begin{align*}
-\mathord{
%
}.
\end{align*}
This is obviously zero if $n = 0$.
Assuming $n > 0$, we apply (\ref{dog1}) to see that it is zero unless
$n = -h-1$,
when the term with $r=-h-2,s=0$ contributes
$\substack{{\color{darkred}{\displaystyle\uparrow}} \\
  {\scriptscriptstyle i}}{\scriptstyle\substack{\lambda\\\phantom{-}}}
$.
This completes the proof of the claim.
Now to establish the relation (\ref{one}),
we vertically compose (\ref{quarry}) on the
bottom with
$\mathord{
%
}$ to obtain the desired relation:
\begin{align*}
\mathord{
%
}.
\end{align*}

The proof of (\ref{two})
follows by a very similar argument;
one first checks that
$$
-\!\!\!\mathord{
%
}
$$
when $h > -2$
then vertically composes on the top with $\displaystyle\mathord{
%
}$.
\end{proof}

\begin{theorem}\label{adjdone}
The following relations hold:
\begin{equation}\label{adjfinal}
\mathord{
%
}.
\end{equation}
\end{theorem}

\begin{proof}
It suffices to prove the left hand relation; the right hand one then
follows using Theorem~\ref{opiso}.
Let $h := \langle h_i,\lambda \rangle$ for short.
If $h \geq 0$ then
\begin{align*}
\mathord{
%
}.
\end{align*}
This completes the proof.
\end{proof}

\section{Cyclicity relations}

At this point, the proof of the Main Theorem is reduced to checking the 
cyclicity relations. We proceed to do this.

\begin{lemma}
The following relation holds:
\begin{align}\label{huddle}
\mathord{
%
}&&\text{if $\langle h_i,\lambda\rangle \geq 1$.}
\end{align}
\end{lemma}

\begin{proof}
The $2$-morphisms on both sides of the desired identity
go from $1_\lambda$ to $E_i F_i 1_\lambda$.
To show that they are equal, we vertically compose them both 
on the top with the isomorphism (\ref{inv2}) to reduce to proving
instead that
$$
\mathord{
%
}.
$$
In these two column vectors of $2$-morphisms, the entries involving dotted
circles are equal thanks to (\ref{first}). It just remains to observe that
$$
\mathord{
%
}.
$$
\end{proof}

\begin{lemma}
Assuming that $i \neq j$,
the following relations hold:
\begin{align}\label{easy1}
\mathord{
%
} &&\text{if $\langle h_i,\lambda\rangle \geq d_{ij}$.}\label{easy2}
\end{align}
\end{lemma}

\begin{proof}
Let $h := \langle h_i,\lambda \rangle$ for short.
First we prove (\ref{easy1}) assuming that $h \leq 0$.
By (\ref{inv1}) and (\ref{inv3}), the following $2$-morphism is invertible:
$$
\mathord{
%
}.
$$
Vertically composing with this on the bottom,
we deduce that the
 relation we are trying to prove is equivalent to the following relations:
\begin{equation}\label{turn}
\mathord{
%
}
 \text{ for $0 \leq n < -h$.}
\end{equation}
To establish the first of these, we pull the $j$-string past the
$ii$-crossing:
\begin{align*}
\mathord{
%
}.
\end{align*}
If $h < 0$ then all the terms on the right hand side vanish thanks to
(\ref{startd}).
If $h = 0$ and $d_{ij} > 0$ everything except for the $r=d_{ij}-1$ term from
the first sum vanishes, and we get $t_{ij} 
\mathord{
%
}$.
Finally if $h = d_{ij} = 0$, we only have the first
term on the right hand side, which
contributes
$t_{ij} \mathord{
%
}$
again thanks to (\ref{everything}), (\ref{ig2}), (\ref{rightpitchfork}) and
(\ref{now}).
This is what we want because:
\begin{align*}
\mathord{
%
}.
\end{align*}
We are just left with the right hand relations from (\ref{turn})
involving bubbles:
\begin{align*}
\mathord{
%
}
+
\left(
\text{a linear comb. of}
\mathord{
%
}.
\end{align*}

The relation (\ref{easy2}) follows by a very similar argument to the
one explained in the previous paragraph;
the first step is to vertically compose on the top with the isomorphism
$$
\displaystyle\mathord{
%
}.
$$

Finally we must prove (\ref{easy1}) if $0 < h < d_{ij}$.
We vertically compose on the bottom with the isomorphism
$\mathord{
%
}$
to reduce to proving that 
\begin{equation}\label{critical}
\mathord{
%
}.
\end{equation}
To see this we apply (\ref{neg}) to transform 
the left hand side into
\begin{align*}
-\mathord{
%
}.
\end{align*}
The first term on the right hand side here vanishes by (\ref{startb}).
Also the terms in the summations are zero unless $r \geq
d_{ij}-h-1$ and $s \geq h$ by (\ref{startb})--(\ref{startd}), hence
we are left just with the
$r=d_{ij}-h-1, s = h$ term, which equals
\begin{align*}
-t_{ij}
\!
\mathord{
%
}.
\end{align*}
This is equal to the right hand  side of (\ref{critical}) thanks to (\ref{hospitalL}).
\end{proof}

\begin{theorem}
The following relations hold for all $i,j \in I$ and $\lambda \in P$:
\begin{align}\label{fourth}
\mathord{
%
}.
\end{align}
\end{theorem}

\begin{proof}
For (\ref{fourth}), we already proved the left hand relation when $\langle h_i,\lambda\rangle
\geq 1$ in (\ref{huddle}). 
Now take this relation with $\lambda$ replaced by $\lambda+\alpha_i$, 
attach leftward caps to the top left and top right strands, then apply (\ref{adjfinal})
to prove the right hand relation when $\langle h_i,\lambda \rangle
\geq -1$. Finally apply Theorem~\ref{opiso} to the cases established so far
to get the right hand relation when
$\langle h_i,\lambda\rangle \leq -1$
and the left hand relation when $\langle h_i,\lambda\rangle \leq 1$.

The proofs of (\ref{easyfinal})--(\ref{hardfinal}) follow by a similar
strategy to the previous paragraph, starting from (\ref{one})--(\ref{two}) and
(\ref{easy1})--(\ref{easy2}).
\end{proof}

The final set of relations (\ref{actress})
needed to complete the proof of the Main Theorem
follow easily from (\ref{fourth})--(\ref{hardfinal}) using also (\ref{adjfinal}).

\end{document}